\documentclass[11pt,reqno]{amsart}      
\usepackage{cite}                       
\usepackage{amssymb,amsthm}             

\usepackage{paralist}                   
\usepackage{calc}                       
\usepackage{graphicx}
\usepackage{latexsym}                   

\numberwithin{figure}{section}          

\numberwithin{equation}{section}        
\renewcommand{\Re}{{\mathbb R}}         
\newcommand{\abs}[1]{\lvert#1\rvert}    
\theoremstyle{plain}
\newtheorem{thm}{Theorem}[section]

\title[Forward Euler for Differential Inclusions]{The
Forward Euler Scheme for Nonconvex Lipschitz Differential Inclusions
Converges With Rate One}
\author{Mattias Sandberg}
\address{Centre of Mathematics for Applications (CMA) c/o Dept of
  Mathematics\\
Box 1053 Blindern\\
NO-0316 Oslo\\
NORWAY
}
\email{mattias.sandberg@cma.uio.no}

\begin{document}
\begin{abstract}
In a previous paper it was shown that the Forward Euler method applied
to differential inclusions where the right-hand side is a Lipschitz
continuous set-valued function with uniformly bounded, compact values,
converges with rate one. The convergence, which was there in the sense
of reachable sets, is in this paper strengthened to the sense of
convergence of solution paths. An improvement of the error constant is
given for the case when the set-valued function consists of a small
number of smooth ordinary functions.
\end{abstract}
\maketitle
\tableofcontents
\section{Introduction}
In this paper we extend the convergence result from
\cite{Sandberg-DI1}. The following differential inclusion is
considered:
\begin{equation}\label{eq:DI}
\begin{split}
x'(t)&\in F\big(x(t)\big),\\
x(0)&=x_0,
\end{split}
\end{equation}
where $x_0\in\Re\sp d$, and $F$ is a function from $\Re\sp d$ to the
compact subsets of $\Re\sp d$. 
In \cite{Sandberg-DI1}, it is shown
that if $F$ is uniformly bounded in the sense that 
\begin{equation}\label{eq:Fbound}
\abs{y}\leq K,\quad \text{for all }y\in\bigcup_{x\in\Re\sp d} F(x),
\end{equation}
and Lipschitz continuous with respect to Hausdorff distance,
\begin{equation}\label{eq:FLip}
\mathcal{H}\big(F(x),F(y)\big) \leq L\abs{x-y},
\end{equation}
then the Forward Euler method converges with rate one. For the
definition of the Hausdorff
distance we need the following notation (as in
\cite{Sandberg-DI1}). We denote by $B$ the closed unit ball in $\Re\sp
d$. The Minkowski sum of two non-empty sets $C,D \subset\Re\sp d$ is
defined by
\begin{equation*}
C+D = \big\{c+d\ |\ c\in C \text{ and } d\in D\big\},
\end{equation*}
the multiplication by a scalar, $\lambda>0$, by
\begin{equation*}
\lambda C = \big\{\lambda c\ |\ c\in C\big\},
\end{equation*}
and the sum of an element $c\in\Re\sp d$ and a set $C$ by
\begin{equation*}
c+C=\big\{c\big\}+C.
\end{equation*}
The Hausdorff distance is given by
\begin{equation*}
\mathcal{H}(C,D) = \inf \big\{\lambda \geq 0\ |\ C \subset D+\lambda B
\text{ and } D \subset C + \lambda B \big\}.
\end{equation*}
We will denote by $\abs{\cdot}$ the Euclidean norm, when applied to a
vector, and the Euclidean operator norm, when applied to a matrix.
We consider solutions $x:[0,T]\rightarrow\Re\sp d$ to the differential inclusion \eqref{eq:DI} in
the finite time interval $[0,T]$. A solution is an absolutely
continuous function which satisfies \eqref{eq:DI} a.e. For the Forward
Euler method we split the interval $[0,T]$ into $N$ parts of equal
length $\Delta t=T/N$. The Forward Euler scheme is given by
\begin{equation}\label{eq:FE}
\begin{split}
\xi_{n+1} &\in \xi_n + \Delta t F(\xi_n), \quad n=0,1, \ldots ,N-1, \\
\xi_0&=x_0.
\end{split}
\end{equation}

The convergence result in \cite{Sandberg-DI1} concerns the reachable sets
\begin{equation*}
\begin{split}
C_n &= \big\{x(n\Delta t)\ |\ x:[0,T]\to \Re^d \text{ solution to
  \eqref{eq:DI}}\big\}, \\
D_n &= \big\{\xi_n\ |\ \{\xi_i\}_{i=0}^N \text{ solution to
  \eqref{eq:FE}}\big\}. 
\end{split}
\end{equation*}
It was shown there that under the assumptions in \eqref{eq:Fbound} and
\eqref{eq:FLip} the following bound holds:
\begin{equation}\label{eq:ConvReachSets}
\max_{0 \leq n \leq N} \mathcal{H}(C_n,D_n) \leq
\Big(Ke^{LT}\big(Kd(d+1)+LT\big) +2Kd \Big)\Delta t.
\end{equation}
This was an extension of the previous first order convergence result
in \cite{Dontchev-Farkhi1989} in the sense that the set-valued
function $F$ did not need to be convex. In \cite{Grammel}, the
non-convex case was presented, although in a different form (see
\cite{Sandberg-DI1}), but there only half-order convergence was
proved. Although the convergence of the reachable sets in
\eqref{eq:ConvReachSets} is what is needed in many situations, e.g.\
in optimal control (see \cite{Sandberg-DI1}), it is weaker than the
convergence of solution paths, the type of convergence used in e.g.\
\cite{Dontchev-Farkhi1989} and \cite{Grammel}. In section \ref{sec:results} we
show that the first-order convergence result for non-convex
differential inclusions can be extended so that it gives convergence
of solution paths. The proof is actually only a minor change of the
proof in \cite{Sandberg-DI1}.
Another weakness with the convergence result in
\eqref{eq:ConvReachSets} is that the constant  depends quadratically
on the dimension. In \cite{Sandberg-DI1} it was shown that this
constant can not be expected  to be smaller than of order $\sqrt d$ in
general. In section \ref{sec:results} a partial improvement is given for the case
where the differential inclusion is a control problem with few control
parameters. 
In section \ref{sec:FD} two results which are needed in the proof of
the theorem involving few control parameters are presented. 

\section{The Results}\label{sec:results} 
We introduce the same set-valued maps that was used in
\cite{Sandberg-DI1}. 
Let $\varphi$ and $\psi$ be
functions from $\Re^d$ into the non-empty compact subsets of $\Re^d$, defined by
\begin{equation*}
\begin{split}
\varphi(x) &= x + \Delta t F(x), \\
\psi(x) &= x + \Delta t\, \text{co}\big(F(x)\big),
\end{split}
\end{equation*}
where co denotes the convex hull.
If $A$ is a subset of $\Re^d$ we define
\begin{equation*}
\varphi(A)= \bigcup_{x\in A}\varphi(x),
\end{equation*}
and similarly for the set-valued maps  $\psi$ and $F$
We will use the following result for convex differential
inclusions. It is taken from \cite{Dontchev-Farkhi1989}, where
it is formulated in a slightly more general setting than the one
presented here.
\begin{thm}\label{thm:ConvexPaths}
Let $F$ be a function from $\Re^d$ into the non-empty compact convex subsets of
$\Re^d$, which satisfies \eqref{eq:FLip} and \eqref{eq:FLip}.
For any solution 
$x : [0,T]\to \Re^d$ to \eqref{eq:DI} there exists a
solution $\{\eta_n\}_{n=0}^N$ to \eqref{eq:FE} such that
\begin{equation}\label{eq:pathapprox} 
\max_{0 \leq n \leq N} |x(n\Delta t)-\eta_n| \leq KLTe^{LT}\Delta t.
\end{equation}
Moreover, for any solution $\{\eta_n\}_{n=0}^N$ to \eqref{eq:FE} there
exists a solution $x:[0,T]\to\Re^d$ to \eqref{eq:DI} such that \eqref{eq:pathapprox} holds.
\end{thm}
The convergence of solution paths to non-convex differential
inclusions is given next.
\begin{thm}\label{thm:convpaths1}
Assume that $x:[0,T]\rightarrow\Re\sp d$ solves (1.1). 
Let $F$ be a function from $\Re^d$ into the non-empty compact  subsets of
$\Re^d$, which satisfies \eqref{eq:FLip} and \eqref{eq:FLip}.
Then there
exists a solution $\{\xi_n\}_{n=0}\sp N$ to \eqref{eq:FE}, such that 
\begin{equation}\label{eq:pathconv}
\max_{0\leq n\leq N}\abs{x(n\Delta t)-\xi_n} \leq K(e\sp{LT}d(d+1) +
2d+ LTe\sp{LT})\Delta t.
\end{equation}  
\end{thm}
\begin{proof}
Let $\{\eta_n\}_{n=0}\sp{N}$ be a solution to the scheme 
\begin{equation*}
\eta_{n+1}\in\psi(\eta_n),\quad \text{for }0\leq n\leq N-1,
\end{equation*}
which satisfies \eqref{eq:pathapprox}. By lemma 2.1 in
\cite{Sandberg-DI1} it follows that the set-valued function
$\text{co}\big(F(x)\big)$  is Lipschitz continuous in the Hausdorff
distance with the same constant as $F(x)$. Therefore Theorem
\ref{thm:ConvexPaths} guarantees the existence of such a solution
$\{\eta_n\}$. 

Let $\varepsilon$ be any positive number, $n$ an integer such that
 $d\leq n\leq N$, and  $\xi_{n-d}$  a point in $\Re\sp d$ such
that 
\begin{equation*}
\eta_n\in \psi\sp d(\xi_{n-d}) + \varepsilon B.
\end{equation*}
Similarly as in the proof of  Theorem 3.4 in \cite{Sandberg-DI1} we have the
following inclusion: 
\begin{equation*}
\psi\big(\psi\sp d(\xi_{n-d})+\varepsilon B\big) \subset \psi\sp
d\big(\varphi(\xi_{n-d})\big) +\big(KLd(d+1)\Delta t\sp 2
+\varepsilon(1+L\Delta t)\big)B
\end{equation*}
Therefore, there must exist an $\xi_{n-d+1}\in\varphi(\xi_{n-d})$, such
that 
\begin{equation*}
\eta_{n+1}\in \psi\sp
d(\xi_{n-d+1}) +\big(KLd(d+1)\Delta t\sp 2
+\varepsilon(1+L\Delta t)\big)B.
\end{equation*}
It follows that there exists a solution $\{\xi_n\}_{n=0}\sp{N-d}$ to
\begin{equation}\label{eq:phievol}
\xi_{n+1}\in\varphi(\xi_n),
\end{equation}
for $0\leq n\leq N-d-1$,
such that 
\begin{equation*}
\eta_{n+d}\in\psi\sp d(\xi_n) +\varepsilon_n B,
\end{equation*}
where 
\begin{equation*}
\begin{split}
\varepsilon_{n+1}&=(1+L\Delta t)\varepsilon_n + KLd(d+1)\Delta t\sp
2,\\
\varepsilon_0&=0.
\end{split}
\end{equation*}
By the proof of Theorem 3.5 in \cite{Sandberg-DI1} it holds that 
\begin{equation*}
\varepsilon_n\leq Ke\sp{LTn/N}d(d+1)\Delta t \leq
Ke\sp{LT}d(d+1)\Delta t.
\end{equation*}
Let us extend the solution $\{\xi_n\}$ up to $n=N$, by letting
$\{\xi_n\}_{n=N-d+1}\sp N$ be any solution to \eqref{eq:phievol} for
$N-d\leq n\leq N-1$. For $d\leq n\leq N$ we have
\begin{equation}\label{eq:xixbound}
\abs{\xi_n-\eta_n} \leq \abs{\xi_n-\xi_{n-d}}+\abs{\eta_n-\xi_{n_d}} \leq
Kd\Delta t +Kd\Delta t + Ke\sp{LT}d(d+1)\Delta t.
\end{equation}
For $0\leq n\leq d$ we have
\begin{equation*}
\abs{\xi_n-\eta_n}\leq \abs{\xi_n-x_0}+\abs{\eta_n-x_0}\leq 2Kd\Delta t.
\end{equation*}
Hence \eqref{eq:xixbound} holds for all $0\leq n\leq N$.
This together with \eqref{eq:pathapprox} gives \eqref{eq:pathconv} 
\end{proof}

Let us now consider the situation where the set-valued function $F$ is
given by 
\begin{equation}\label{eq:disccontrol1}
F(x)=\{f_i(x)\}_{i=1}\sp M,
\end{equation}
and where we have smoothness, in the sense that there exists a
constant $S>0$, such that
\begin{equation}\label{eq:disccontrol2}
\abs{f_i(x)-f_i(z)-f'_i(z)(x-z)}\leq S\abs{x-z}\sp 2,
\end{equation}
for all $1\leq i\leq M$ and $x,z\in\Re\sp d$. By \eqref{eq:Fbound} we
have
\begin{equation}\label{eq:disccontrol3}
\abs{f_i(x)}\leq K, \quad\text{for all $1\leq i\leq M$ and $x\in\Re\sp
  d$.}
\end{equation} 
Let us assume that we also have the following bound on the
Jacobians:
\begin{equation}\label{eq:disccontrol4}
\abs{f'_i(x)}\leq L,\quad\text{for every }x\in\Re\sp d.
\end{equation}
Under these assumptions we have
\begin{thm}\label{thm:fullydiscrete}
Assume that \eqref{eq:disccontrol1}, \eqref{eq:disccontrol2},
\eqref{eq:disccontrol3}, and \eqref{eq:disccontrol4} hold.
Assume that $M\geq d+1$, and that $x:[0,T]\rightarrow\Re\sp d$ solves \eqref{eq:DI}. Then
there exists a solution $\{\xi_n\}_{n=0}\sp{N}$ to \eqref{eq:FE}, such
that
\begin{multline}\label{eq:pathM}
\max_{0\leq n\leq N}\abs{x(n\Delta t)-\xi_n} \leq
\big(e\sp{LT}(KLT+K(8M-10))+2K(M-1)\big)\Delta t\\ 
+e\sp{LT}\big(KL(M-1)(M-2) + 2KL\frac{(M-1)\sp3-(M-1)}{3}(1+L\Delta
t)\sp{M-3}\\
+ 2SK\sp2\frac{M(M-1)(2M-1)}{3L}\big)\Delta t\sp2.
\end{multline} 
\end{thm}
\begin{proof}
This proof follows the same basic lines as the proof of Theorem
\ref{thm:convpaths1}. Hence we assume that we have a solution
$\{\eta_n\}_{n=0}\sp N$ to the scheme
\begin{equation*}
\eta_{n+1}\in\psi(\eta_n),\quad\text{for }0\leq n\leq N-1,
\end{equation*}
which satisfies \eqref{eq:pathapprox}. Let $\varepsilon$ be any
positive number, $n$ an integer such that $M-1\leq n\leq N$, and
$\xi_{n-M+1}$ a point in $\Re\sp d$ such that  
\begin{equation*}
\eta_n \in \text{co}\big(\varphi\sp{M-1}(\xi_{n-M+1})\big)+\varepsilon B.
\end{equation*} 
Similarly as in the proof of Theorem 3.4 in \cite{Sandberg-DI1} we
have the following inclusion:
\begin{equation*}
\psi\big(\text{co}(\varphi\sp{M-1}(\xi_{n-M+1}))+\varepsilon B\big)
\subset
\psi\big(\text{co}(\varphi\sp{M-1}(\xi_{n-M+1}))\big)+\varepsilon(1+L\Delta t)B.
\end{equation*}
By Theorems \ref{thm:psiconvexhull} and \ref{thm:coco} we have
\begin{multline*}
\psi\big(\text{co}(\varphi\sp{M-1}(\xi_{n-M+1}))\big)+\varepsilon(1+L\Delta
t)B \subset 
\bigcup_{x\in\varphi(\xi_{n-M+1})} \text{co}\big(\varphi\sp{M-1}(x)\big)\\ 
+\big(\varepsilon(1+L\Delta t) +(8M-10)KL\Delta t\sp 2\\ 
+(2KL\sp2\frac{(M-1)\sp3-(M-1)}{3}(1+L\Delta t)\sp{M-3}\\ 
+2SK\sp 2
\frac{M(M-1)(2M-1)}{3})\Delta t\sp3\big)B=:\big(\varepsilon(1+L\Delta
t) + C_1\Delta t\sp2 + C_2\Delta t\sp 3\big)B.
\end{multline*}
Therefore, there must exist a $\xi_{n-M+2}\in\varphi(\xi_{n-M+1})$, such
that 
\begin{equation*}
\eta_{n+1}\in \text{co}\big(\varphi\sp{M-1}(\xi_{n-M+2})\big) +
\big(\varepsilon(1+L\Delta t) +C_1\Delta t\sp 2 + C_2\Delta t\sp 3\big)B.
\end{equation*}
Hence there exists a solution $\{\xi_n\}_{n=0}\sp{N-M+1}$ to 
\begin{equation*}
\begin{split}
\xi_{n+1}&\in\varphi(\xi_n),\\
\xi_0=x_0,
\end{split}
\end{equation*} 
for $0\leq n\leq N-M$, such that
\begin{equation*}
\eta_{n+M-1}\in \text{co}\big(\varphi\sp{M-1}(\xi_n)\big) +
\varepsilon_n B,
\end{equation*}
where
\begin{equation*}
\begin{split}
\varepsilon_{n+1}&=(1+L\Delta t)\varepsilon_n+ C_1\Delta t\sp
2+C_2\Delta t\sp 3,\\
\varepsilon_0&=KL(M-1)(M-2)\Delta t\sp2.
\end{split}
\end{equation*}
From this we have that 
\begin{equation*}
\epsilon_n = (1+L\Delta t)\sp{n+1}\varepsilon_0 + (C_1\Delta
t\sp2+C_2\Delta t\sp3)\big(1+(1+L\Delta t)+\cdots+(1+L\Delta t)\sp n\big).
\end{equation*}
For $0\leq n\leq N-M$ we have $(1+L\Delta t)\sp n\leq e\sp{LT}$ and 
\begin{equation*}
1+(1+L\Delta t)+\cdots+(1+L\Delta t)\sp n \leq
\frac{e\sp{LT}-1}{L\Delta t}\leq\frac{e\sp{LT}}{L\Delta t}.
\end{equation*}
Hence
\begin{equation*}
\varepsilon_n \leq KLe\sp{LT}(M-1)(M-2)\Delta
t\sp2+\frac{e\sp{LT}}{L}(C_1\Delta t+C_2\Delta t\sp2).
\end{equation*}
As in the proof of Theorem \ref{thm:convpaths1} we can extend
$\{\xi_n\}$ up to $N$ and have 
\begin{equation*}
\abs{\xi_n-\eta_n}\leq KLe\sp{LT}(M-1)(M-2)\Delta
t\sp2+\frac{e\sp{LT}}{L}(C_1\Delta t+C_2\Delta t\sp2) + 2K(M-1)\Delta t.
\end{equation*} 
This together with \eqref{eq:pathapprox} gives us \eqref{eq:pathM}. 
\end{proof}

\section{The fully discrete case}\label{sec:FD}
We present here two results, Theorems \ref{thm:coco} and
\ref{thm:psiconvexhull}, that are useful for the proof of Theorem
\ref{thm:fullydiscrete}. We will use the following well-known result,  
the Carath\'eodory Theorem:
\begin{thm}\label{thm:Caratheodory}
The convex hull of an arbitrary subset $A$ of $\Re^d$ is given by
\begin{equation*}
\text{co}(A)=\Big\{\sum_{i=1}^{d+1}\lambda_i a_i\ |\ a_i\in A,
\lambda_i \geq 0, \sum_{i=1}^{d+1} \lambda_i=1\Big\}.
\end{equation*}
\end{thm}
For a proof, see \cite{Berger}.

\begin{thm}\label{thm:coco}
Assume that \eqref{eq:disccontrol1}, \eqref{eq:disccontrol2},
\eqref{eq:disccontrol3}, and \eqref{eq:disccontrol4} hold, and that
$M \geq d+1$. Then
\begin{equation}\label{eq:convincl}
\text{co}\big(\varphi\sp{M}(x_0)\big)\subset\cup_{x\in\varphi(x_0)}\text{co}\big(\varphi\sp{M-1}(x)\big)+RB,
\end{equation}
where
\begin{multline*}
R=(8M-10)KL\Delta t\sp2 +
\big(2KL\sp2\frac{(M-1)\sp3-(M-1)}{3}(1+L\Delta t)\sp{M-3}\\ 
+SK\sp2\frac{M(M-1)(2M-1)}{3}\big)\Delta t\sp3.
\end{multline*}
\end{thm}
\begin{proof}
We start by introducing the notation
\begin{equation*}
b_i=f_i(x_0),\ A_i=f'_i(x_0),\quad \text{for }0\leq i\leq M.
\end{equation*}
To begin with, we will make the assumption that the functions $f_i$
are given by
\begin{equation}\label{eq:linearized}
f_i(x)=b_i+A_i(x-x_0).
\end{equation}
Afterwards, we will consider the general case.
For simplicity, we will prove \eqref{eq:convincl} for the case where
$M=d+1$. The general result follows directly from this. We will also
assume that $x_0=0$, to simplify the presentation.

\emph{Step 1.} 
Every point $x$ in $\varphi\sp{d+1}(x_0)$ is given by
\begin{equation}\label{eq:pointinphidplusone}
x=x_0+\Delta t f_{i_1}(x_0)+\Delta t f_{i_2}(x_1)+\cdots+\Delta t
f_{i_{d+1}}(x_d), 
\end{equation}
where $i_j\in\{1,\ldots,d+1\}$, for all $1\leq j\leq d+1$, and $x_1$,
$x_2$,..., are defined recursively by
\begin{equation*}
x_{n+1}=x_n+\Delta t f_{i_n}(x_n), \quad \text{for } n=0,\ldots,d.
\end{equation*} 
Since we now assume that the functions $f_i$ are given by \eqref{eq:linearized},
we have
\begin{equation}\label{eq:discevolonestep}
x_{n+1}=x_n+\Delta t b_{i_n}+\Delta t A_{i_n}x_n.
\end{equation}
When we sum the terms in \eqref{eq:pointinphidplusone} under the
consideration of \eqref{eq:discevolonestep}, we see that
\begin{equation*}
x=x_0+\sum_{r=1}\sp{d+1} \sum_{1\leq
  k_1<k_2<\ldots<k_r\leq d+1} A_{i_{k_r}}A_{i_{k_{r-1}}}\cdots
  A_{i_{k_2}}b_{i_{k_1}}\Delta t\sp r.
\end{equation*} 
Let $\tilde x$ be the approximation of $x$, where all terms of power
three or larger in $\Delta t$ have been dropped, i.e.
\begin{equation*}
\tilde x=x_0+\sum_{r=1}\sp{2} \sum_{1\leq
  k_1<k_2<\ldots<k_r\leq d+1} A_{i_{k_r}}A_{i_{k_{r-1}}}\cdots
  A_{i_{k_2}}b_{i_{k_1}}\Delta t\sp r.
\end{equation*} 
With the bounds on $A_i$ and $b_i$ from \eqref{eq:disccontrol3}and
\eqref{eq:disccontrol4},  we have that 
\begin{multline*}
\abs{\tilde x-x}\leq \sum_{r=3}\sp{d+1} {d+1 \choose r}
L\sp{r-1}K\Delta t\sp r\\
=\frac{K}{L}\sum_{r=0}\sp{d+1}{d+1 \choose r}
(L \Delta t)\sp r - \frac{K}{L}\big(1+(d+1)L\Delta t+ \frac{d(d+1)}{2}L\sp 2 \Delta
t\sp 2\big)\\
=\frac{K}{L}(1+L\Delta t)\sp{d+1}- \frac{K}{L}\big(1+(d+1)L\Delta t+ \frac{d(d+1)}{2}L\sp 2 \Delta
t\sp 2\big).
\end{multline*} 
With a Taylor expansion of the function $f(x)=(1+x)\sp{d+1}$ around
$x=0$, we establish that 
\begin{equation*}
(1+x)\sp{d+1}\leq 1+(d+1)x +
  \frac{d(d+1)}{2}x\sp2+\frac{d\sp3-d}{6}(1+x)\sp{d-2}x\sp 3,
\end{equation*}
and hence
\begin{equation}\label{eq:higherorderbound}
\abs{\tilde x-x}\leq KL\sp2\frac{d\sp3-d}{6}(1+L\Delta t)\sp{d-2}\Delta
t\sp 3.
\end{equation}

\emph{Step 2.} We now consider the convex combination
\begin{equation}\label{eq:twoconvcomb}
(1-\frac{1}{N})x\sp1+\frac{1}{N}x\sp2,
\end{equation}
where $x\sp1$ and $x\sp2$ are two elements in $\varphi\sp{d+1}(x_0)$,
such that in the expression in \eqref{eq:pointinphidplusone} for
$x\sp1$, none of the indices $i_n$ equals one, while for $x\sp2$, $N$
of the indices equal one.   We will see how well the convex combination
in \eqref{eq:twoconvcomb} can be represented by another convex
combination,
\begin{equation}\label{eq:twoconvcomb2}
(1-\frac{1}{N})\tilde x\sp1+\frac{1}{N}\tilde x\sp2,
\end{equation}
where $\tilde x\sp1$ and $\tilde x\sp2$ both have precisely one index
$i_n$ equal to one in the expression in
\eqref{eq:pointinphidplusone}. 

Pick any $n\in\{1,2,\ldots,d+1\}$. Assume that we define $\tilde
x\sp1$ by changing the index $i_n$ (denote $i_n=k$) in the expression
\eqref{eq:pointinphidplusone} for $x\sp1$ to one. Then 
\begin{multline}\label{eq:x1diff}
\tilde x\sp1-x\sp1 = \Delta t(b_1-b_{i_n}) + \Delta
t\sp2\big((A_1-A_{i_n})(b_{i_1}+b_{i_2}+\cdots+b_{i_{n-1}})\\
+(A_{i_{n+1}}+\cdots +A_{i_{d+1}})(b_1-b_{i_n}) \big) + \ \text{higher
order terms.}
\end{multline} 
Simiarly, we define $\tilde x\sp2$ by exchanging all the indices 
for which $i_n=1$ to $i_n=k$. Then the first order term in \eqref{eq:twoconvcomb2}
is the same as in \eqref{eq:twoconvcomb}. The second order term in
the difference $\tilde x\sp 1-x\sp1$ in \eqref{eq:x1diff} is bounded
in magnitude by $2KLd\Delta t\sp 2$. 
Since this bound holds independently of which of the indices was
changed, it follows that the second order term in the difference
$\tilde x\sp 2-x\sp2$ is bounded in magnitude by $2NKLd\Delta
t\sp2$. Hence the difference in the second order term between the
convex combinations in \eqref{eq:twoconvcomb}  and in
\eqref{eq:twoconvcomb2} is bounded by 
\begin{equation*}
(1-\frac{1}{N})2KLd\Delta t\sp2+\frac{1}{N}2NKLd\Delta
  t\sp2=(4-\frac{2}{N})KLd\Delta t\sp2 \leq 4KLd\Delta t\sp2.
\end{equation*}
In step 1, we established that the sum of all terms of order higher
than or equal to three in $\Delta t$ for every element in
$\varphi\sp{d+1}(x_0)$ is bounded as in
\eqref{eq:higherorderbound}. We thereby have
\begin{equation}\label{eq:convcomberror} 
\big|(1-\frac{1}{N})(\tilde x\sp1-x\sp1)+\frac{1}{N}(\tilde x\sp
  2-x\sp2)\big| \leq 4KLd\Delta t\sp 2 + KL\sp2\frac{d\sp3-d}{3}(1+L\Delta
  t)\sp{d-2}\Delta t\sp3. 
\end{equation}

\emph{Step 3.} Let $z$ be any element in
$\text{co}\big(\varphi\sp{d+1}(x_0)\big)$. By the Carath\'eodory
Theorem (Theorem \ref{thm:Caratheodory}), we have that there exists a
$G\leq d+1$ and points and constants $x\sp i\in\varphi\sp{d+1}(x_0)$ and
$\alpha_i> 0$, for $1\leq i\leq G$, such that 
\begin{equation}\label{eq:zconvcomb}
z=\sum_{i=1}\sp{G}\alpha_i x\sp i, 
\end{equation} 
and 
$\sum_{i=1}\sp{G} \alpha_i=1$. 
We can then write
\begin{equation*}
z=x_0+\Delta t\sum_{i=1}\sp{d+1}\gamma_ib_i+\ \text{higher
order terms,} 
\end{equation*}
where $\gamma_i\geq 0$ for $1\leq i\leq d+1$ and
$\sum_{i=1}\sp{d+1}\gamma_i=d+1$. It must hold that at least one of
the coefficients $\gamma_i\geq 1$. For simplicity, let us assume that
$\gamma_1\geq 1$.
We will now present an algorithm
which gives us an approximation of $z$ in the form of a convex
combination of points in $\varphi\sp{d+1}(x_0)$ which all have one
index $i_n=1$ in the formula \eqref{eq:pointinphidplusone}. We also
give an error bound of this approximation. 

Let $I\subset\{1,2,\ldots,G\}$ be the index set of all the points
$x\sp i$ in \eqref{eq:zconvcomb} for which none of the indices $i_1$
in the formula for the points in $\varphi\sp{d+1}(x_0)$ equals
one. Let $J$ be a set which consists of the weights $\alpha_i$
corresponding to elements in $I$, i.e.\  $i\in I$ if and only if
$\alpha_i\in J$.

If $I=\emptyset$, we already have what we are aiming for. Let us
therefore assume that $I$ is nonempty, and for simplicity that $i=1$
is one of the elements therein. Take one element in
$\{1,2,\ldots,d+1\}\setminus I$, such that the corresponding element
$x\sp i$ in the convex combination \eqref{eq:zconvcomb} is of the form 
\begin{equation*}
x\sp i = x_0+\Delta t(Nb_1+\cdots)+\ \text{higher order terms,}
\end{equation*}
with $N\geq 2$. Such an element must exist, since $\gamma_1\geq
1$. For simplicity, let us assume that $x\sp2$ is one such
element. 
We may write
\begin{equation*}
\frac{1}{\alpha_1+\alpha_2}(\alpha_1 x\sp1 + \alpha_2 x\sp 2) = x_0 +
\Delta t(kb_1 +\cdots) +\ \text{higher order terms.}
\end{equation*}
One of the two following cases must hold:
\begin{enumerate}
\item $k > 1$. When this is the case we rewrite as follows:
\begin{multline}\label{eq:convcombrewritten}
\alpha_1 x\sp1+\alpha_2 x\sp2=\alpha_1 x\sp 1 +
\frac{\alpha_1}{N-1}x\sp2 +
\big(\alpha_2-\frac{\alpha_1}{N-1}\big)x\sp2\\
=\frac{N}{N-1}\alpha_1\big(\big(1-\frac{1}{N}\big)x\sp1+\frac{1}{N}x\sp
2\big)+\big(\alpha_2-\frac{\alpha_1}{N-1}\big)x\sp2.
\end{multline}
Since $k>1$ we have that $\alpha_2-\alpha_1/(N-1)$ is positive.
By the result from step 2, we have the approximation result in
\eqref{eq:convcomberror}, with some points $\tilde x\sp1$ and $\tilde
x\sp2$, both being of the form
\begin{equation*}
x_0+\Delta t(b_1+\cdots)+\ \text{higher order terms}.
\end{equation*}
Together with \eqref{eq:convcombrewritten}, this implies that
\begin{equation}\label{eq:kappacomb}
\sum_{i=1}\sp G \alpha_i x\sp i = \alpha_1\tilde x\sp1 +
\frac{\alpha_1}{N-1}\tilde x\sp2 +
\big(\alpha_2-\frac{\alpha_1}{N-1}\big)x\sp2 + \sum_{i=3}\sp G
\alpha_i x\sp i + \kappa,
\end{equation}
where
\begin{equation}\label{eq:kappabound}
\abs{\kappa}\leq \frac{N}{N-1}\alpha_1\big(4KL d\Delta t\sp 2 + KL\sp2\frac{d\sp3-d}{3}(1+L\Delta
  t)\sp{d-2}\Delta t\sp3\big).
\end{equation}

\item $k\leq 1$. In this case we rewrite as follows:
\begin{multline*}
\alpha_1 x\sp 1+\alpha_2 x\sp 2 = (N-1)\alpha_2 x\sp 1 + \alpha_2
x\sp2 + \big(\alpha_1-(N-1)\alpha_2\big)x\sp1\\
=N\alpha_2\big(\big(1-\frac{1}{N}\big)x\sp1+\frac{1}{N}x\sp
2\big) + \big(\alpha_1-(N-1)\alpha_2\big)x\sp1.
\end{multline*}
Since $k\leq 1$ we have that $\alpha_1-(N-1)\alpha_2$ is
nonnegative. Similarly as in case (1), we have 
\begin{equation}\label{eq:kappacomb2}
\sum_{i=1}\sp{G}\alpha_i x\sp i = (N-1)\alpha_2\tilde x\sp1 +
\alpha_2\tilde x\sp 2 + \big(\alpha_1-(N-1)\alpha_2\big)x\sp1 +
\sum_{i=3}\sp G \alpha_ix\sp i + \kappa,
\end{equation}
where
\begin{equation}\label{eq:kappabound2}
\abs\kappa\leq N\alpha_2\big(4KLd\Delta t\sp 2 + KL\sp2 \frac{d\sp3-d}{3}(1+L\Delta
  t)\sp{d-2}\Delta t\sp3\big).
\end{equation}
\end{enumerate}

If case (1) holds we let 
\begin{equation*}
\hat z = \alpha_1\tilde x\sp 1 + \frac{\alpha_1}{N-1}\tilde x\sp 2 +
\big(\alpha_2-\frac{\alpha_1}{N-1}\big)x\sp 2 + \sum_{i=3}\sp G
\alpha_i x\sp i,
\end{equation*}
and remove $i=1$ from $I$ and $\alpha_1$ from $J$. If case (2) holds
we let
\begin{equation*}
\hat z=(N-1)\alpha_2\tilde x\sp1 +
\alpha_2\tilde x\sp 2 + \big(\alpha_1-(N-1)\alpha_2\big)x\sp1 +
\sum_{i=3}\sp G \alpha_ix\sp i,
\end{equation*}
and replace $\alpha_1$ with $(N-1)\alpha_2$ in $J$.
We then iterate the process above with $z$ replaced by $\hat z$, and
the new sets $I$ and $J$. We continue this process until the sets $I$
and $J$ are empty, and we have an approximation $\tilde z$ of $z$ of
the form
\begin{equation}\label{eq:tildez}
\tilde z = \sum_{i}\tilde\alpha_i \tilde x\sp i,
\end{equation}
where every $\tilde x\sp i$ is of the form
\begin{equation*}
\tilde x\sp i = x_0 + \Delta t(mb_1 + \cdots) +\ \text{higher order terms,}
\end{equation*}
with $m$ an integer greater than or equal to  one. 
We note that the factor $N\alpha_2$, appearing in the error estimate
in \eqref{eq:kappabound2} is the same as the weights of the new points
$\tilde x\sp1$ and $\tilde x\sp2$ in \eqref{eq:kappacomb2}. The same
holds also for case (1), with \eqref{eq:kappabound} and
\eqref{eq:kappacomb}. Since the total weight of the points that have
been changed can not be larger than one, we have that 
\begin{equation}\label{eq:zerr1}
\abs{\tilde z - z} \leq 4KLd\Delta t\sp2+KL\sp2 \frac{d\sp3-d}{3}(1+L\Delta
t)\sp{d-2}\Delta t\sp3.
\end{equation}

\emph{Step 4.} We now approximate the point $\tilde z$ in \eqref{eq:tildez}
by a point in $\varphi\sp d(x_0+\Delta t f_1(x_0))$. Consider any
point $\tilde x\sp i$ in the convex combination  in
\eqref{eq:tildez}. When $\tilde x\sp i$ is computed by equation
\eqref{eq:pointinphidplusone} we know that at least one of the indices must
equal one. Let us assume that $i_n=1$. We denote by $\bar x\sp i$ the
element in $\varphi\sp d(x_0+\Delta t f_1(x_0))$ we obtain by
switching the indices $i_1$ and $i_n=1$ in the expression for $\tilde
x\sp i$ in \eqref{eq:pointinphidplusone}. We then have the difference
\begin{multline*}
\bar x\sp i -\tilde x\sp i= \Delta
t\sp2\big((A_{i_2}+\cdots+A_{i_{n-1}})(b_{i_n}-b_{i_1})
+A_{i_1}(b_{i_2}+\cdots+b_{i_n})\\
-A_{i_n}(b_{i_1}+\cdots+b_{i_{n-1}})\big)+\ \text{higher order terms.}
\end{multline*}
We may get a bound for the magnitude of the difference $\bar x\sp
i-\tilde x\sp i$ by using the bounds on $\abs{A_i}$ and $\abs{b_i}$ from
\eqref{eq:disccontrol3} and \eqref{eq:disccontrol4} and the bound on the higher order terms of $\bar
x\sp i$ and $\tilde x\sp i$ in \eqref{eq:higherorderbound}. We get the
largest possible difference if $n=d+1$:
\begin{equation*}
\abs{\bar x\sp i-\tilde x\sp i} \leq (4d-2)KL\Delta t\sp 2 +  KL\sp
2\frac{d\sp 3-d}{3}(1+L\Delta t)\sp{d-2}\Delta t\sp 3.
\end{equation*}  
Now let 
\begin{equation*}
\bar z = \sum_i\tilde\alpha_i\bar x\sp i.
\end{equation*}
Since $\sum_i\tilde\alpha_i=1$, we therefore have that
\begin{equation}\label{eq:zerr2}
\abs{\bar z-\tilde z} \leq (4d-2)KL\Delta t\sp 2 +  KL\sp
2\frac{d\sp 3-d}{3}(1+L\Delta t)\sp{d-2}\Delta t\sp 3.
\end{equation}

\emph{Step 5.} We now consider the contribution to the error from the
fact that we may not have \eqref{eq:linearized}, but instead the
functions $f_i$ satisfy  \eqref{eq:disccontrol2}. Denote by $P$ the
set $\text{co}\big(\varphi\sp{d+1}(x_0)\big)$ when it is computed
using  \eqref{eq:linearized}. The set
$\text{co}\big(\varphi\sp{d+1}(x_0)\big)$ in the general case
satisfies the inclusion
\begin{equation*}
\text{co}\big(\varphi\sp{d+1}(x_0)\big) \subset P+rB,
\end{equation*}
where 
\begin{multline*}
r=SK\sp 2\Delta t\sp 3 + 4 SK\sp 2 \Delta t\sp 3 +\cdots + d\sp2
SK\sp2\Delta t\sp3\\ 
= SK\sp2 \frac{d(d+1)(2d+1)}{6}\Delta t\sp3.
\end{multline*}
An error of size $r$ is made also when 
\begin{equation*}
\cup_{x\in\varphi(x_0)}\text{co}\big(\varphi\sp d(x)\big)
\end{equation*}
is approximated using \eqref{eq:linearized}. 
This, together with \eqref{eq:zerr1}, \eqref{eq:zerr2}, and $M=d+1$
gives \eqref{eq:convincl}. 
\end{proof}

\begin{thm}\label{thm:psiconvexhull}
Assume that \eqref{eq:disccontrol1}, \eqref{eq:disccontrol2},
\eqref{eq:disccontrol3}, and \eqref{eq:disccontrol4} hold. Then 
\begin{equation}\label{eq:thirdorderincl}
\psi\big(\text{co}(\varphi\sp{M-1}(z))\big) \subset
\text{co}\big(\varphi\sp M(z)\big) +
SK\sp2\frac{M(M-1)(2M-1)}{3}\Delta t\sp3B.
\end{equation}
\end{thm}
\begin{proof}
To begin with, let us assume that 
\begin{equation}\label{eq:simplefuncs}
f_i(x)=b_i+A_i(x-z),\quad \text{for }1\leq i\leq M.
\end{equation}
We will show that this implies that 
\begin{equation}\label{eq:psico}
\psi\big(\text{co}(\varphi\sp{M-1}(z))\big) \subset
\text{co}\big(\varphi\sp M(z)\big).
\end{equation}
Let $v$ be any unit vector in $\Re\sp d$ and consider the function
\begin{equation*}
(x,\alpha_1,\ldots,\alpha_M)\mapsto v\cdot(x+\sum_{i=1}\sp M \alpha_i f_i(x)),
\end{equation*}
over the set
\begin{equation*}
\big\{x,\alpha_1,\ldots,\alpha_M : x\in \text{co}(\varphi\sp{M-1}(z)),
\alpha_i\in\Re, \alpha_i\geq 0, \sum_{i=1}\sp M \alpha_i=1\big\}.
\end{equation*}
The function is continuous, and the set is compact, and hence there is
a maximizer
\begin{equation*}
x\sp*,\alpha_1\sp*,\ldots,\alpha_M\sp*.
\end{equation*}
We note two things:
\begin{enumerate}
\item Since the function
\begin{equation*}
x\mapsto v\cdot(x+\Delta t\sum_{i=1}\sp M \alpha_i\sp* f_i(x))
\end{equation*}
is linear, 
its maximum over $\text{co}\big(\varphi\sp{M-1}(z)\big)$ is attained at a point in
$\varphi\sp{M-1}(z)$. 
\item Since the function
\begin{equation*}
(\alpha_1,\ldots,\alpha_M)\mapsto v\cdot(x\sp*+\Delta t\sum_{i=1}\sp M \alpha_i f_i(x\sp*))
\end{equation*}
is linear, its maximum over 
\begin{equation*}
\big\{\alpha_1,\ldots,\alpha_M:\alpha_i\in\Re,\alpha_i\geq0, \sum_{i=1}\sp{M}\alpha_i=1\big\}
\end{equation*}
is attained at a point where one of the $\alpha_i$:s are one.
\end{enumerate}
Since $v$ can be any element in $\Re\sp d$,  these facts imply that
\begin{equation}\label{eq:inclusionequality}
\text{co}\big(\psi\big(\text{co}(\varphi\sp{M-1}(z))\big)\big) =
\text{co}\big(\varphi\sp N(z)\big),
\end{equation} 
which implies \eqref{eq:psico}.

Now let $b_i=f_i(z)$ and $A_i=f'_i(z)$. Let us denote by $P$ the set
in \eqref{eq:inclusionequality} when the functions in
\eqref{eq:simplefuncs} are used. 
By \eqref{eq:disccontrol2} and \eqref{eq:disccontrol3} we have that 
\begin{equation}\label{eq:incl1}
\psi\big(\text{co}(\varphi\sp{M-1}(z))\big) \subset P + rB,
\end{equation} 
where
\begin{multline*}
r=SK\sp 2\Delta t\sp 3 + 4 SK\sp 2 \Delta t\sp 3 +\cdots + (M-1)\sp2
SK\sp2\Delta t\sp3\\ 
= SK\sp2 \frac{M(M-1)(2M-1)}{6}\Delta t\sp3.
\end{multline*}
Similarly,
\begin{equation}\label{eq:incl2}
P\subset \text{co}\big(\varphi\sp{M}(z)\big)+rB.
\end{equation}
The inclusion \eqref{eq:thirdorderincl} follows by \eqref{eq:incl1} and \eqref{eq:incl2}. 
\end{proof}

\bibliographystyle{plain}

\bibliography{references} 


\end{document}